\documentclass[a4paper,11pt,reqno]{amsart}

\usepackage{amsthm,amsmath,amsfonts,amssymb}
\usepackage{color}

\usepackage{todonotes}
\usepackage{empheq}
\usepackage{cases}

\usepackage{hyperref}

\definecolor{grey}{rgb}{.7,.7,.7}

\usepackage{color}

\newcommand{\e}{\varepsilon}

\newcommand{\N}{\mathbb{N}}
\newcommand{\de}{\partial}

\renewcommand{\phi}{\varphi}

\renewcommand{\div}{\mathop{\mathrm{div}}}
\renewcommand{\epsilon}{\varepsilon}

\def\R{\mathbb{R}}

\def\H{\mathcal{H}}

\theoremstyle{plain}
\newtheorem{thm}{Theorem}[section]
\newtheorem*{thm*}{Theorem}
\newtheorem{lem}[thm]{Lemma}
\newtheorem{prop}[thm]{Proposition}

\theoremstyle{definition}
\newtheorem{defin}[thm]{Definition}

\theoremstyle{remark}
\newtheorem{rem}[thm]{Remark}

\title[Weighted Cheeger sets are domains of isoperimetry]{Weighted Cheeger sets\\ are domains of isoperimetry}

\author{Giorgio Saracco}
\address{Department Mathematik, Universit\"at Erlangen-N\"urnberg, Cauerst. 11, 91058 Erlangen - Germany}
\email{saracco@math.fau.de}

\thanks{G. Saracco has been supported by the 2016 INDAM-GNAMPA project \textit{Variational problems and geometric measure theory in metric spaces}}



\begin{document}

\begin{abstract}
We consider a generalization of the Cheeger problem in a bounded, open set $\Omega$ by replacing the perimeter functional with a Finsler-type surface energy and the volume with suitable powers of a weighted volume. We show that any connected minimizer $A$ of this weighted Cheeger problem such that $\H^{n-1}(A^{(1)} \cap \de A)=0$ satisfies a relative isoperimetric inequality. If $\Omega$ itself is a connected minimizer such that $\H^{n-1}(\Omega^{(1)} \cap \de \Omega)=0$, then it allows the classical Sobolev and $BV$ embeddings and the classical $BV$ trace theorem. The same result holds for any connected minimizer whenever the weights grant the regularity of perimeter-minimizer sets and $\Omega$ is such that $|\de \Omega|=0$ and $\H^{n-1}(\Omega^{(1)} \cap \de \Omega)=0$.
\end{abstract}

 \hspace{-3cm}
 {
 \begin{minipage}[t]{0.6\linewidth}
 \begin{scriptsize}
 \vspace{-3cm}
 This is a pre-print of an article published in \emph{Manuscripta Math.}. The final authenticated version is available online at: http://dx.doi.org/10.1007/s00229-017-0974-z
 \end{scriptsize}
\end{minipage} 
}

\maketitle

\section{Introduction} 

The celebrated Cheeger problem, firstly proposed in \cite{Cheeger}, consists in searching for sets $A \subseteq \R^n$ minimizing the ratio
\begin{equation}\label{eq:intro1}
\inf_{E\subseteq \Omega} \frac{P(E)}{|E|}\,,
\end{equation}
where $\Omega$ is a given open bounded set. This problem firstly arose in connection with estimates on the first eigenvalue of the Laplacian (see for instance \cite{KF}) and then appeared in several other contexts such as the capillarity problem (see for instance \cite{LeoSar2016}), image reconstruction (see for instance \cite{ROF}) and quantum waveguides (see for instance \cite{KreKri, KrePra, LeoPra2015}). Two nice overviews can be found in \cite{Leo, Parini2011}. The classical problem has then been tweaked in different directions: it has been proposed in the Gaussian setting (see \cite{CMjN}) or modified to the non-local fractional perimeter (see \cite{BPL}). We are here interested in generalizing the local, Euclidean version in a way that embraces both the weighted problem proposed in \cite{Carlier, Caselles2009, Hassani2005, Ionescu2005}, and the one proposed in \cite{Pratelli} where suitable powers of the volume are considered. Specifically, given an open bounded set $\Omega \subseteq \R^n$, we deal with the following minimization problem
\begin{equation}\label{eq:CG}
\inf_{E\subseteq \Omega} \frac{\int_{\de^*E} g(x, \nu_E(x))\, d\, \mathcal{H}^{n-1}(x)}{\left|\int_E f\, dx \right|^{1/\alpha}}\,,
\end{equation}
where $\alpha \in \left[1, \frac{n}{n-1} \right)$, $f(x)$ is a positive $L^\infty$ function and $g(y,v)$ is a scalar function, lower semi-continuous in $(y,v) \in \R^n \times \R^n$, convex and positively $1$-homogeneous in $v$ and such that for some $C>0$ one has
\begin{equation*}
\frac{1}{C} |v| \le g(y,v) \le C|v|\,,
\end{equation*}
for all $y \in \R^n$, $v\in \R^n$. First, in Proposition \ref{existence} we prove the existence of minimizers for the proposed problem \eqref{eq:CG}. Second, we show that each connected minimizer, $A$, such that $\H^{n-1}(A^{(1)}\cap \de A)=0$ is a \emph{domain of isoperimetry}, by which we mean that it supports a relative isoperimetric inequality, i.e. there exists a constant $k(A)$ such that for all $E\subset A$
\begin{equation}\label{eq:intro2}
\min\left \{|E|; |A \setminus E|\right \}^{n-1} \leq k P(E; A)^n\,.
\end{equation}
Actually, in Theorem \ref{prop:traceinequality}, which we recall below, we prove a stronger inequality.
\begin{thm*}
Let $A$ be a connected minimizer of \eqref{eq:intro1} such that $\H^{n-1}(A^{(1)}\cap \de A)=0$. Then, there exists a positive constant $K$ depending only on $A$ such that 
\begin{equation}\label{eq:intro3}
	\min\{ P(E; \de A),P(A \setminus E; \de A) \} \leq K\, P(E; A)\,.
\end{equation}
\end{thm*}
This fact, up to our knowledge, has not been proved even for the classical Cheeger problem. We deem it important as for open, bounded and connected sets, equation \eqref{eq:intro2} is known to be equivalent to the existence of the classical Sobolev and $BV$ embedding operators (see \cite[Section 5.2.3 and Section 9.1.7]{Mazya2011}) while \eqref{eq:intro3} coupled with $P(A) = \H^{n-1}(\de A)$ to the existence of the $BV$ trace operator (see \cite[Section 9.6.4]{Mazya2011}). Then, we are able to infer the following theorem (see Theorem \ref{thm:main}).
\begin{thm*}
Let $\Omega$ be connected and such that $\H^{n-1}(\Omega^{(1)}\cap \de \Omega)=0$. Suppose it is a generalized Cheeger set in itself, i.e. it is an open, bounded set that realizes the infimum in \eqref{eq:CG} staged in $\Omega$ itself. Then, there exists a positive constant $k$ depending only on $\Omega$ such that for all $u \in W^{1,p}(\Omega)$
\[
\|u\|_{L^{\frac{np}{n-p}}(\Omega)} \leq k \|u\|_{W^{1,p}(\Omega)}\,,
\]
and for all $u\in BV(\Omega)$
\[
\|u\|_{L^1(\Omega)} \leq k \|u\|_{BV(\Omega)}\,,
\]
Moreover, there exists a linear continuous operator (the trace) $T: BV(\Omega) \to L^1(\de \Omega)$ such that for all $u\in BV(\Omega)$ continuous up to $\de \Omega$, one has $T(u) = u_{|\de \Omega}$.
\end{thm*}
Notice that without further hypotheses on $\Omega$ and on the weights $f,g$ the previous theorem can not be extended to any connected minimizer $A$ in $\Omega$, as one would need to prove that $A$ is open and that $\H^{n-1}(A^{(1)} \cap \de A)=0$. Whenever one can dispose of the regularity theory, a condition on $\Omega$ that yields the openness of minimizers is given by $|\de \Omega|=0$. Again exploiting the regularity of minimizers, one can show that $\H^{n-1}(\Omega^{(1)} \cap \de\Omega)=0$ is enough to have the equality $\H^{n-1}(A^{(1)} \cap \de A)=0$.

This kind of results is not completely new and it should not surprise that perimeter minimizers exhibit ``good isoperimetric properties''; a result in similar fashion was obtained in \cite{DS} for quasi-minimizers with respect to any variation.
%
%

\section{Preliminaries}

We start the section recalling some basic facts of the theory of sets of finite perimeter (for more details one can refer to \cite{AFP}) and then we introduce the definitions of weighted volume (see Definition \ref{def:weightedvol}), weighted perimeter (see Definition \ref{def:weightedper}) and prove an isoperimetric inequality between these weighted quantities (see Proposition \ref{prop:weighisop}).

For a Borel set $E \subset \R^n$ we will denote by $|E|$ its $n$-dimensional Lebesgue measure and by $P(E)$ its perimeter in the sense of De Giorgi, i.e.
\begin{equation}\label{eq:perimeter}
P(E):=\sup \left\{ \int_{\R^n} \chi_E(x) \div h(x)\, dx\,: h\in C^1_c(\R^n;\, \R^n)\,, \|h\|_{\infty} \leq 1\right\}\,.
\end{equation}
If $P(E)<\infty$ we say that $E$ is a set of finite perimeter. In this case one has that the perimeter of $E$ agrees with the total variation $|D\chi_{E}|(\R^n)$ of the vector-valued Radon measure $D\chi_{E}$. This allows us to define the relative perimeter $P(E; \Omega) = |D\chi_{E}|(\Omega)$ for any pair of Borel sets $E, \Omega \subset \R^n$.
\\
We recall that a point $x\in \R^n$ is said to be of density $\beta \in [0,1]$ for a Borel set $E\subset \R^n$ if the limit
\[
\theta(E)(x) := \lim_{r\to0^+} \frac{|E\cap B_r(x)|}{|B_r(x)|}
\]
exists and equals $\beta$. The set of points of density $\beta \in [0,1]$ for $E$ is denoted by $E^{(\beta)}$. We also define the essential boundary of $E$ as $\de^{e}E := \R^{n}\setminus (E^{(0)}\cup E^{(1)})$. Finally, for a set of finite perimeter $E$ we say that a point $x\in \de^e E$ belongs to the reduced boundary $\de^* E$ if the following limit
\[
\lim_{r\to 0^+} - \frac{D\chi_E(B_r(x))}{| D\chi_E|(B_r(x))}
\]
exists and belongs to $\mathbb{S}^{n-1}$. In this case we denote such a limit by $\nu_E(x)$ and call it the outer normal of $E$ at $x$. With these definitions in mind we can state De Giorgi's and Federer's structure theorems.
\begin{thm}[De Giorgi Structure Theorem]\label{thm:DeGiorgi}
Let $E$ be a set of finite perimeter. Then,
\begin{itemize}
\item[(i)] $\de^{*}E$ is countably $\H^{n-1}$-rectifiable in the sense of Federer (see \cite{FedererBOOK});
\item[(ii)] for all $x\in \de^{*}E$, $\chi_{E_{x,r}} \to \chi_{H_{\nu_E(x)}}$ in $L^{1}_{loc}(\R^{n})$ as $r\to 0^{+}$, where $E_{x,r}:= (E-x)/r$ and $H_{\nu_{E}(x)}$ denotes the half-space through $0$ whose outer normal is $\nu_{E}(x)$;
\item[(iii)] for any Borel set $\Omega$, $P(E;\Omega) = \H^{n-1}(\Omega \cap \de^{*}E)$, thus in particular $P(E)=\H^{n-1}(\partial^* E)$;
\item[(iv)] $\int_{E}\div g = \int_{\de^{*}E} g\cdot \nu_{E}\, d\H^{n-1}$ for any $g\in C^{1}_{c}(\R^{n};\R^{n})$.
\end{itemize}
\end{thm}

\begin{thm}[Federer Structure Theorem]\label{thm:Fed}
Let $E$ be a set of finite perimeter. Then, $\de^* E \subset E^{(1/2)} \subset \de^e E$ and one has
\[
\H^{n-1}\left(\de^e E \setminus \de^* E\right)=0 \,.
\]
\end{thm}

We now give the definitions of weighted volume and weighted perimeter which we will later use to define the weighted Cheeger problem. 

\begin{defin}[Weighted volume]\label{def:weightedvol}
Let $E$ be a Lebesgue measurable set in $\R^n$ and $f\in L^\infty(\R^n)$ be a positive weight. We define the weighted Lebesgue measure of $E$ as
\begin{equation}\label{eq:weightedLeb}
|E|_f := \int_E f\, dx\,.
\end{equation}
\end{defin}

In view of Theorem \ref{thm:DeGiorgi} (iii), one can give as an alternate definition of perimeter of a set $E$ in $\Omega$
\[
P(E;\Omega) =  \int_{\Omega \cap \de^{*}E} \, d\, \H^{n-1}(x)\,.
\]
We here choose to mimic this one rather than the more classical \eqref{eq:perimeter} to define the weighted perimeter.

\begin{defin}[Weighted perimeter]\label{def:weightedper}
Let $E$ be a Borel set in $\R^n$ and let $g : \R^n \times \R^n \to \R$ be a lower semi-continuous function, convex and positively $1$-homogeneous in the second variable for which it exists $C>0$ such that
\[
\frac{1}{C}|v| \le g(x,v) \le C|v|\,,
\] 
for all $(x,v) \in \R^n \times \R^n$. We define the perimeter of $E$ weighted through $g$ in a Borel set $\Omega\subset \R^{n}$ as
\begin{equation}\label{eq:weightedperimeter}
P_g(E; \Omega):= \int_{\Omega \cap \de^{*}E} g(x, \nu_E(x))\, d\, \H^{n-1}(x) =\H^{n-1}_g(\Omega \cap \de^{*}E).
\end{equation}
We set $P_g(E) = P_g(E;\R^{n})$.
\end{defin}

\begin{rem}\label{prop:boundPg}
As a straightforward consequence of the hypotheses on $g$, the weighted perimeter $P_g(E; \Omega)$ of a set $E$ in $\Omega$ has both a lower bound and an upper bound in terms of the classical perimeter $P(E; \Omega)$ given by
\[
\frac{1}{C} P(E; \Omega) \leq P_g(E; \Omega) \leq C P(E; \Omega)\,.
\]
\end{rem}

\begin{prop}\label{prop:weighisop}
There exists a constant $c = c(f,g)$ such that
\[
|E|_f^{n-1} \leq cP_g(E)^n\,,
\]
for all Lebesgue measurable subsets $E \subset \R^n$ of finite volume.
\end{prop}

\begin{proof}
Since $f\in L^\infty(\R^n)$ one can bound $|E|_f$ with $\|f\|_{L^\infty}|E|$. Then, the claimed inequality follows from the classical isoperimetric inequality and the lower bound of Remark \ref{prop:boundPg}.
\end{proof}

\section{Generalized weighted Cheeger sets}

\begin{defin}[Generalized weighted Cheeger set]
Let $\Omega$ be an open, bounded set in $\R^n$. Let $\alpha \in \left[1, 1^*\right)$, where $1^* := n/(n-1)$. We define the $(f,g,\alpha)$-Cheeger constant of $\Omega$ as
\begin{equation}\label{eq:genCconst}
h^\alpha_{f,g}(\Omega) := \inf_{E\subseteq \Omega} \frac{P_g(E)}{|E|^{1/\alpha}_f}\,,
\end{equation}
where the infimum is sought amongst all non empty subsets of $\Omega$ with finite perimeter. We shall denote by  $\mathcal{C}_{f,g}^\alpha(\Omega)$  the family of minimizers of \eqref{eq:genCconst}. 
\end{defin}


Note that the triplet $(1,1,1)$ corresponds to the classical Cheeger problem, of which an overview can be found in \cite{Leo, Parini2011}, while the triplet $(f,g(x)|v|,1)$ corresponds to the weighted version of the problem, which was dealt with in \cite{Carlier, Caselles2009, Hassani2005, Ionescu2005} (up to choosing more regular $g$ in some of those papers), and finally the triplet $(1,1, \alpha)$ corresponds to the version dealt with in \cite{Pratelli}. For all these cases, existence is known and proved in the abovementioned papers. 

Existence is as well retained in this more general case. The proof is fairly standard and uses the hypotheses on $g$ to exploit Reshetnyak's lower semi-continuity theorem (see \cite[Theorem 2.38]{AFP} or the original paper \cite{Res}) on the functional $P_g$.

\begin{prop}\label{existence}
For an open, bounded set $\Omega$ the family $\mathcal{C}^\alpha_{f,g}(\Omega)$ is not empty, i.e. there exists at least one set $A\subset \Omega$ such that
\[
h^\alpha_{f,g}(\Omega) = \frac{P_g(A)}{|A|_f^{1/\alpha}}\,.
\]
\end{prop}

\begin{proof}
Being $\Omega$ open, it is trivial that $h(\Omega) <\infty$: one can take any ball $B_r \subset \subset \Omega$ and easily check that $P_g(B_r)/|B_r|_f^{1/\alpha}$ is finite.  Let $\{A_k\}_k$ be a minimizing sequence for (\ref{eq:genCconst}). Since $\Omega$ is bounded, we have that $\{A_k\}_k$ is an equibounded family in $L^1$. Let now be $\epsilon > 0$: it has to exist an index $\bar{k}$ such that
\[
\left| h (\Omega) -\frac{P_g(A_k)}{|A_k|^{1/\alpha}_f} \right| \leq \epsilon\,,
\]
for all $k\ge \bar{k}$. Exploiting the lower bound given by Proposition \ref{prop:boundPg} we get
\[
|D\chi_{A_k}| (\mathbb{R}^n) = P(A_k) \leq C P_g(A_k) \leq C\left(\e +h(\Omega)\right)(\|f\|_\infty|\Omega|)^{1/\alpha}\,,
\]
thus $\{A_k\}_k$ is an equibounded family in the $BV$ norm. Thus, up to subsequences, it converges in the $L^1$ topology and pointwise almost everywhere to a function $u$. Hence, it is a characteristic function of a set $A \subset \Omega$. Using the lower semi-continuity theorem of Reshetnyak (see \cite[Theorem 2.38]{AFP}), for the functional $P_g$ and the $L^1$ convergence of $f\chi_{A_k}$ to $f\chi_A$, we infer that $A \in \mathcal{C}^\alpha_{f,g}(\Omega)$, as soon as we prove $|A|_f>0$. Argue by contradiction and suppose it equals zero. Hence, $|A_k|_f \to 0$. Fix a ball $B_{r_k}$ of the same Euclidean volume of $A_k$, then
\[
\frac{P_g(A_k)}{|A_k|_f^{1/\alpha}} \ge \frac{1}{C\|f\|_\infty^{1/\alpha}}\frac{P(A_k)}{|A_k|^{1/\alpha}} \ge \frac{1}{C\|f\|_\infty^{1/\alpha}}\frac{P(B_{r_k})}{|B_{r_k}|^{1/\alpha}} = \frac{n\omega_n^{1-\frac{1}{\alpha}}r_k^{n-1-\frac{n}{\alpha}}}{C\|f\|_\infty^{1/\alpha}}  \to \infty\,,
\]
against the fact that $A_k$ is a minimizing sequence.
\end{proof}

We shall now focus on the isoperimetric properties of sets $A$ in $\mathcal{C}^\alpha_{f,g}(\Omega)$. We aim to prove that any connected minimizer $A$ of \eqref{eq:genCconst} such that $\H^{n-1}(A^{(1)}\cap \de A)=0$ is a domain of isoperimetry i.e. \eqref{eq:intro2} holds. More precisely, we shall show that \eqref{eq:intro3} holds for such minimizers and that this inequality implies \eqref{eq:intro2}. We start by proving this last implication in the next lemma.

\begin{lem}\label{lem:relperimeter}
Let $A$ be a connected set. If there exists $k=k(A)>0$ such that \eqref{eq:intro3} holds, i.e.
\begin{equation*}
\min\{ P(E; \de A),P(A \setminus E; \de A) \} \leq k\, P(E; A) \qquad \forall\, E \subset A,
\end{equation*}
then there exists $K = K(A)>0$ such that \eqref{eq:intro2} holds, i.e.
\begin{equation*}
\min\{|E|;|A \setminus E|\}^{\frac{n-1}{n}} \leq K\, P(E;A), \qquad \forall E \subset A.
\end{equation*}
\end{lem}

\begin{proof}
Since $E\subset A$ and $P(E; A) = P(A\setminus E; A)$ we have
\begin{align}
P(E; \de A) &= P(E) - P(E;A)\,, \nonumber \\
P(A \setminus E; \de A) &= P(A \setminus E) - P(E; A)\,. \nonumber
\end{align}
Plugging these identities in ~(\ref{eq:intro3}) and then exploiting the isoperimetric inequality give
\begin{equation*}
\frac{k+1}{n\omega_n^{1/n}}P(E; A) \geq \frac{1}{n\omega_n^{1/n}} \min \{ P(E), P(A \setminus E) \} \geq \min \left \{ |E|^{\frac{n-1}{n}} , |A \setminus E|^{\frac{n-1}{n}}\right \}\,,
\end{equation*}
which is the claim for $K(A) = (k(A)+1)^{-1} n\omega_n^{1/n}$.
\end{proof}

\begin{thm}\label{prop:traceinequality}
Let $A \in \mathcal{C}^\alpha_{f,g}(\Omega)$ be connected and such that $\H^{n-1}(A^{(1)} \cap \de A)=0$. Then there exists $K=K(A)>0$ such that \eqref{eq:intro3} holds with $K$.
\end{thm}

\begin{proof}
Being $A$ connected, by \cite[comments to Chapter~9]{Mazya2011} one has that \eqref{eq:intro3}, which we recall to be
\begin{equation*}
	\min\{ P(E; \de A),P(A \setminus E; \de A) \} \leq k\, P(E; A),
	\end{equation*} 
is equivalent to 
\begin{equation}\label{eq:BVtracecondition}
	\sup_{x\in \partial A} \lim_{\rho \to 0^+} \sup \left \{ \frac{P(E; \de A)}{P(E; A)} \Bigg | E\subset A \cap B_\rho (x) \right \} < +\infty.
\end{equation}
In order to prove this one, assume by contradiction that for some sequence of points $x_j\in \de A$, of radii $\rho_j \searrow 0$ and of sets $E_j \subset A \cap B_{\rho_j}(x_j)$, one has
\begin{equation}
\frac{P(E_j; \de A)}{P(E_j; A)} \xrightarrow[j\to \infty]{} \infty.
\end{equation}
We then have
\begin{equation}\label{eq:contraddizione}
0 \leq \frac{P(E_j; A)}{P(E_j)}  \leq \frac{P(E_j; A)}{P(E_j; \de A)}  \xrightarrow[j\to \infty]{} 0,
\end{equation}
which tells us that $P(E_j; A) = o(P_g(E_j))$ due to Remark \ref{prop:boundPg}. 

Let us now consider the competitor $A_j = A \setminus E_j$. Start noticing that for any $F\subset A$ one has
\begin{equation}\label{eq:splitperimeter}
P_g(F) =  P_g(F ; \de A)+P_g(F; A)\,.
\end{equation}
Moreover, since $g$ is positively $1$-homogeneous in the second variable and $A_j$ and $E_j$ are complement sets in $A$ we have as well $P_g(A_j; A) = P_g(E_j; A)$. Then,
\[
P_g(A_j) = P_g(A_j; \de A) +P_g(E_j; A)\,.
\]
Add and subtract $P_g(E_j; \de A)$ in the previous one, exploit that, analogously to the Euclidean perimeter,
\[
P_g(A) = P_g(A_j; \de A) + P_g(E_j; \de A) - 2\H_g^{n-1}(\de^* A_j \cap (A^{(1)}\cap \de A)) 
\]
(see for instance \cite[Theorem 16.3 and Exercise 16.6]{Maggi}) and finally use the hypothesis $\H^{n-1}(A^{(1)} \cap \de A)=0$ (which implies the same holds for $\H_g^{n-1}$) to get
\[
P_g(A_j) = P_g(A) -P_g(E_j; \de A) +P_g(E_j; A)\,.
\]
Using \eqref{eq:splitperimeter} with $F=E_j$, we observe the following equality
\[
P_g(A_j)= P_g(A) -P_g(E_j)+2P_g(E_j; A)\,.
\]
Hence, as $P(E_j; A) = o(P_g(E_j))$, we have
\begin{equation}
h^\alpha_{f,g}(A) \leq \frac{P_g(A_j)}{|A_j|^{1/\alpha}_f} = \frac{P_g(A) - P_g(E_j) + o(P_g(E_j))}{(|A|_f - |E_j|_f)^{1/\alpha}}\,.
\end{equation}
In order to proceed, notice that $1-x^\beta \leq (1-x)^\beta$ for $x\in [0,1]$ and $\beta \in (0,1)$. Thus, by using the isoperimetric inequality stated in Proposition \ref{prop:weighisop}, the chain of inequalities goes on as
\begin{equation}
h^\alpha_{f,g}(A) \leq \frac{P_g(A) - k|E_j|_f^{\frac{n-1}{n}} +o\left (|E_j|_f^{\frac{n-1}{n}}\right )}{|A|_f^{1/\alpha} - |E_j|_f^{1/\alpha}} < \frac{P_g(A)}{|A|_f^{1/\alpha}} = h^\alpha_{f,g}(A)\,
\end{equation}
for $j>>1$ and $k>0$, since $|E_j|_f \to 0$ as $f \in L^1(A)$ and $|E_j|\to 0$ and $\alpha < 1^*$. Hence, a contradiction.
\end{proof}

In the next lemma we show that whenever a set $A$ admits a relative isoperimetric inequality, then $\H^{n-1}(A^{(0)}\cap \de A)=0$.

\begin{lem}\label{lem:PH}
Let $A$ be an open bounded set of finite perimeter such that it supports a relative isoperimetric inequality. Then, $A^{(0)} \cap \partial A = \emptyset$.
\end{lem}

\begin{proof}
Argue by contradiction and suppose that $A^{(0)} \cap \partial A \neq \emptyset$: fix a point $x_0 \in A^{(0)} \cap \partial A$. For any $r$ set
\[
m(r)  := \left| A \cap B_r(x_0)\right|.
\]
It is well known that (see for instance \cite[Example 13.4]{Maggi}), for almost every $r$, one has
\[
m'(r) = P(B_r(x_0); A).
\]
By the relative isoperimetric inequality on $A$, for $r$ small enough, it follows
\begin{equation}\label{eq:diff}
\frac {m'(r)}{m(r)^{\frac{n-1}{n}}} \ge c.
\end{equation}
By integrating \eqref{eq:diff} between $\rho$ and $2\rho$ one obtains
\[
m^{\frac{1}{n}}(2\rho) - \frac{c}{n} \rho \geq m^{\frac{1}{n}} (\rho) \geq 0, 
\]
hence
\[
m(2\rho) \geq \left(\frac{c}{2n}\right)^n 2^n\rho^n\,.
\]
However, this contradicts the assumption $x_0 \in A^{(0)}$.  
\end{proof}

By combining Theorem \ref{prop:traceinequality}, Lemma \ref{lem:relperimeter} and Lemma \ref{lem:PH} we immediately get our main theorem.

\begin{thm}\label{thm:main}
Let $\Omega$ be connected and such that $\H^{n-1}(\Omega^{(1)}\cap \de \Omega)=0$. Suppose it is a generalized Cheeger set in itself, i.e. it is an open, bounded set that realizes the infimum in \eqref{eq:CG}, staged in $\Omega$ itself.  Then, there exists a positive constant $k$ depending only on $\Omega$ such that for all $u \in W^{1,p}(\Omega)$
\[
\|u\|_{L^{\frac{np}{n-p}}(\Omega)} \leq k \|u\|_{W^{1,p}(\Omega)}\,,
\]
and for all $u\in BV(\Omega)$
\[
\|u\|_{L^1(\Omega)} \leq k \|u\|_{BV(\Omega)}\,,
\]
Moreover, there exists a linear continuous operator (the trace) $T: BV(\Omega) \to L^1(\de \Omega)$ such that for all $u\in BV(\Omega)$ continuous up to $\de \Omega$, one has $T(u) = u_{|\de \Omega}$.
\end{thm}

\begin{proof}
Being $\Omega$ a connected minimizer with $\H^{n-1}(\Omega^{(1)}\cap \de \Omega)=0$, by Theorem \ref{prop:traceinequality} we know that \eqref{eq:intro3} holds and so does \eqref{eq:intro2} via Lemma \ref{lem:relperimeter}. Thus the first part of the claim on the Sobolev and $BV$ embeddings follows immediately by \cite[Section 5.2.3 and Section 9.1.7]{Mazya2011}), being $\Omega$ open, bounded and connected. 

On the other hand, since $\Omega$ is a minimizer it must have finite weighted perimeter, thus finite perimeter by Remark \ref{prop:boundPg}. Then, as both $\Omega^{(1)}\cap \de \Omega$ (by hypothesis) and $\Omega^{(0)}\cap \de \Omega$ (by Lemma \ref{lem:PH}) have $\H^{n-1}$ null measure, Federer's theorem implies that $P(\Omega) =\H^{n-1}(\de \Omega)$. Then, the second claim  immediately follows by \cite[Section 9.6.4]{Mazya2011}.
\end{proof}

\begin{rem}
Notice that $\Omega \in \mathcal{C}^\alpha_{f,g}(\Omega)$ in general does not imply that the perimeter and the $(n-1)$-dimensional Hausdorff measure of the topological boundary agree so that the hypothesis $\H^{n-1}(\Omega^{(1)} \cap \de \Omega)=0$ can not be dropped. For instance, consider the unit disk $B_1$ and let $C_0^{\e} \subset [-\e, \e]\times \{0\}$. Iteratively take a decreasing sequence $C^\e_i$ of compact subsets of $C^\e_0$ obtained at each step $i$ by removing the $2^{i-1}$ open segments $S_j^i$ for $j=1,\dots,2^{i-1}$ of length $\H^1(S^i_j) = 2^{1-2i}\H^1(C^\e_{i-1})$ and placed in the middle of each closed segment of $C^\e_{i-1}$. The limit set $C^\e = \lim_i C^\e_i$ is a Cantor set of positive measure. Let $\delta>0$ be fixed, set
\[
f_\delta(x) = \begin{cases}
1-\sqrt{1- \left(|x|-\delta\right)^2} & \text{if }x\in (-\delta, \delta),\\
0 & \text{otherwise,}
\end{cases}
\]
and
\[
F_\delta =\{ (x,y)\in \R^2\,:\, |x|\le \delta,\ |y|\leq f_\delta(x) \}\,.
\]
For $i\in \N$ we set $\delta_{i} = 2^{-2i}\H^{1}(C^{\e}_{i-1})$ and $m^{i}_{j}$ as the midpoint of $S^{i}_{j}$, then define  
\[
F^{\e} = \bigcup_{i\in \N}\bigcup_{j=1}^{2^{i-1}} F^{i}_{j}\,,
\]
where $F^{i}_{j} = m^{i}_{j} + F_{\delta_{i}}$. For $\e<<1$, the set $\Omega \setminus \overline{F^\e}$ can be shown to be the minimizer of the (classical) Cheeger problem in itself and clearly $\H^1(\de \Omega) > P(\Omega)$ as its topological boundary contains the Cantor set $C^\e$ of positive measure. The full computations are contained in the forthcoming paper \cite{LeoSar20162}.
\end{rem}

\begin{rem}
One would like to have the same result of Theorem \ref{thm:main} for any connected minimizer $A$. To achieve this, one must prove that such minimizers $A$ are open sets such that $\H^{n-1}(A^{(1)}\cap \de A)=0$. In the standard case, or even in the case of a triplet $(1,1,\alpha)$ it is easily shown that the hypothesis $|\de \Omega|=0$ implies the openness of minimizers, and that the hypothesis $\H^{n-1}(\Omega^{(1)}\cap \de \Omega)=0$ implies the same holds with $A$ in place of $\Omega$. This is a consequence of the regularity theory which can be employed since $A$ is a perimeter-minimizer at fixed volume. Therefore $\de A\cap \Omega$ is analytic possibly except for a closed singular set whose Hausdorff dimension is at most $n-8$ (where $n$ denotes the dimension of $\Omega \subset \R^n$). We briefly show these two facts.

Suppose that $|\de \Omega|=0$. Then,  $A$ is Lebesgue equivalent to its interior points $A^\circ$. Indeed take a sequence $\Omega_j$ of open subset relatively compact in $\Omega$ such that $\Omega = \cup_j \Omega_j$. One has
\[
|A\setminus A^\circ| \le |\de A \cap \de \Omega| +\sum_j |(A\setminus A^\circ)\cap \Omega_j| =0\,,
\]
where the first term is zero since $|\de \Omega|=0$, and the second is zero as $\de A\cap \Omega_j$ is an analytic hyper-surface (except for a negligible closed set). 

Suppose now that $\H^{n-1}(\Omega^{(1)}\cap \de \Omega)=0$. Then on one hand, $\de A \cap \de \Omega$ can not have points of density $1$ for $A$ (as it does not even have points of density $1$ for $\Omega$). On the other hand, by regularity one has $\H^{n-1}(A^{(1)}\cap (\de A \cap \Omega))=0$.

The previous reasoning can be applied whenever one dispose of a regularity theory of isoperimetric sets with densities. As of now, such results are available only when dealing with the same weight at both perimeter and volume, i.e. $f=g$ under $C^{k,\gamma}$ regularity (see \cite{Mor03}) or being lower semi-continuous and bounded from above and below (see \cite{CP15}).
\end{rem} 

\section*{Conflict of Interest}
The author declares that he has no conflict of interest.

\section*{Acknowledgements}
The author would like to thank Gian Paolo Leonardi for proposing the problem, for the careful reading of the manuscript and his useful comments.

\bibliographystyle{plain}

\bibliography{isoperimetricCheeger}

\end{document}